\documentclass[11pt,oneside,english]{amsart}
\usepackage{amssymb}

\makeatletter \theoremstyle{plain}
 \newtheorem{thm}{Theorem}[section]

 \newtheorem{prop}[thm]{Proposition}
 \newtheorem{corollary}[thm]{Corollary}
 \numberwithin{equation}{section} 
 \numberwithin{figure}{section} 
 \theoremstyle{plain}
 \theoremstyle{definition}
 \newtheorem{defn}[thm]{Definition}
 \newtheorem{rem}[thm]{Remark}
\newcommand{\A}{{{\mathcal A}}}

\newcommand{\calS}{{{\mathcal S}}}
\newcommand{\calG}{{{\mathcal G}}}
\newcommand{\h}{{{\mathfrak H}}}

\newcommand{\E}{{{\mathcal S}}}
\newcommand{\calT}{{{\mathcal T}}}
\newcommand{\V}{{{\mathcal V}}}
\newcommand{\C}{{{\mathbb C}}}

\newcommand{\R}{{{\mathbb R}}}
\newcommand{\calR}{{{\mathcal R}}}
\newcommand{\J}{{{\mathbb J}}}

\usepackage{babel}
\makeatother

\begin{document}

\title [Straight Ruled surfaces]{Straight ruled surfaces in the heisenberg group}

\author{Ioannis D. Platis}

\address {Department of Mathematics,  University of Crete, Knossos Avenue, 71409 Heraklion Crete, Greece.}


\email{jplatis@math.uoc.gr}

\begin{abstract}
We generalise a result of Garofalo and Pauls: a horizontally minimal smooth surface embedded in the Heisenberg group is locally a straight ruled surface, i.e. it consists of straight lines tangent to a horizontal vector field along a smooth curve. We show additionally that any horizontally minimal surface is locally contactomorphic to the complex plane.
\end{abstract}

\maketitle

\section{Introduction}

A ruled surface in the Heisenberg group $\h$ is a surface which is foliated by  geodesics of the Carnot--Caratheodor\'y metric $d_{cc}$ in $\h$. These geodesics are the rulings of the surface, and when they are Euclidean straight lines we call the ruled surface {\it straight}. The class of Heisenberg ruled surfaces is the analogue of its Euclidean  counterpart; it is a classical theorem of elementary differential geometry of surfaces that ruled surfaces embedded in $\R^3$ have vanishing Gaussian curvature $K$ and are locally isometric to the  plane. Moreover, every sufficient small portion of  a surface which is locally isometric to the plane is a generalised cylinder, or a generalised cone or a tangent developable, see for instance \cite{DC}. 

A smooth surface $\calS$ embedded in $\h=\R^3$ inherits a sub--Riemannian structure from the one of $(\h,d_{cc})$; this is described by the horizontal normal vector field $\nu_\calS$ on the surface. Points of the surface where $\nu_{\calS}$ can not be defined are called characteristic and the set of these points form the characteristic locus of $\calS$. The pull--back of the contact form $\omega$ of $\h$ defines a $1-$form $\omega_\calS$ in $\calS$ and two surfaces $\calS$ and $\tilde\calS$ are called locally contactomorphic if there exists a local diffeomorphism $f:\calS\to\tilde\calS$ away from the characteristic loci so that $f^*\omega_{\tilde\calS}=\lambda\omega_\E$. Such contactomorphisms between surfaces are the sub--Riemannian analogues of local isometries in the Euclidean case. A notion of mean curvature, the horizontal mean curvature $H^h$, is defined in non characteristic points of $\calS$ in terms of the derivatives of the components of $\nu_\calS$: if $X$ and $Y$ are the horizontal vector fields of $\h$ and $\nu_\calS=\nu_1X+\nu_2Y$ then
$$
H^h=X\nu_1+Y\nu_2.
$$  
Surfaces with vanishing horizontal mean curvature are called $H-$minimal.
In \cite{GP}, Garofalo and Pauls proved the following result concerning surfaces in $\h$ which are graphs of functions over the $xy-$plane (Corollary 5.3):

\medskip

\noindent {\bf Theorem.} {\it If $\calS$ is a portion of a $C^2-$ surface $S$ which is a graph of a function over the $xy-$plane in $\h$ with non characteristic points, then it is $H-$minimal if and only if it is a piece of a ruled surface whose  rulings are straight lines (i.e. astraight ruled surface).}  

\medskip 

In this article, we consider arbitrary smooth surfaces (not necessarily graphs) embedded in $\h$, see Section \ref{sec:surfaces} for details. Our main theorem is the following version of Garofalo--Pauls' result: 
\begin{thm}\label{thm:main}
{\it Straight ruled surfaces} have zero horizontal mean curvature and are all locally contactomorphic to the complex plane. Moreover, if a surface $\calS$ has everywhere zero horizontal mean curvature, then every sufficiently small portion of $\calS$ comprising only of non characteristic points is a straight ruled surface.
\end{thm}  
The method of proof is elementary and may be summarised as follows. The kernel of the induced contact form $\omega_\E$ of $\E$ is the vector field $\mathbb{J}\nu_\E\in T(\E)$, where $\mathbb{J}$ is the natural complex structure acting naturally on the horizontal bundle of $\h$. Therefore, away from characteristic points there is a foliation of $\E$ by horizontal surface curves (the horizontal flow of $\E$). It is proved (see Proposition \ref{prop:mean-sec} and also \cite{CDPT}) that the horizontal mean curvature $H^h(p)$ at a non characteristic point of $\S$ is equal to $\kappa_s(p')$, where $p'$ is $pr_\C(p)$ and $\kappa_s(p')$ is the signed curvature of the plane curve which is the projection to $\C$ of the leaf of the horizontal flow passing from $p$. Now , a surface in $\h$ which is locally contactomorphic to the complex plane must have $H^h=0$ (see Proposition \ref{prop:contact-minimal}) and straight ruled surfaces share this property (Proposition \ref{prop:straight-contact}.) On the other hand, in an $H-$minimal surface all  projected curves of the horizontal flow are straight lines and the only option 
for a sufficiently small portion of $\calS$ containing an arbitrary non characteristic point $p$ is to be a straight ruled surface, (see proof of Theorem \ref{thm:main} in Section \ref{sec:ruled}).

Next, we give two examples of classes of smooth surfaces with empty characteristic locus. The first one is that of the horizontal tangent developables, see Section \ref{sec:hor-developables}. The second is that of surfaces which besides empty characteristic locus, also have closed induced contact form. We prove the following:

\medskip

\noindent {\bf Proposition \ref{prop:closed}.} {\it Smooth surfaces in $\h$ with empty characteristic locus and closed induced contact form are exactly the planes which are perpendicular to the complex plane} $\C$. 

\medskip

This result is comparable to the next Bernstein type Theorem, see \cite{DGNP3}:

\medskip

\noindent {\bf Theorem.}  {\it The only stable $C^2-$minimal entire graphs in $\h$ with empty characteristic locus, are the vertical planes}
$$
\Pi=\{(x,y,t)\in\h\;|\;ax+by=c,\;\; a,b,c\in\R\}.
$$ 

\medskip

Stability here is in the sense that every compact subset of a surface $\E$ minimises the horizontal area (or perimeter) up to the second order; for details see \cite{DGNP2} and \cite{DGNP3}.

There is a quite large bibliography in $H-$minimal surfaces. Illustratively, a characterisation of  minimal surfaces in terms of a subelliptic PDE may be found in \cite{P};  Benstein type problems are addressed (and solved) in  \cite{DGNP1}, \cite{DGNP2}, \cite{DGNP3}, \cite{GP}, \cite{I}. More general results may be also found in \cite{CHMY}.

This paper is organised as follows. In Section \ref{sec:prel} we discuss in brief the Heisenberg group and its sub--Riemannian geometry. In Section \ref{sec:elements} we set up the environment of our work, discussing regular surfaces in $\h$ and elements of their horizontal geometry. In Section \ref{sec:ruled} we discuss straight ruled surfaces and prove our main theorem and finally, surfaces with empty characteristic locus are presented in Section \ref{sec:empty}.

\section{The Heisenberg Group}\label{sec:prel}
The material of this section is standard; we refer the reader for instance to \cite{CDPT}, \cite{G} and \cite{KR}. 
The Heisenberg group $\h$  is the set $\R^2\times\R$ with the group law
$$
(x,y,t)*(x',y',t')=\left(x+x',y+y',t+t'+2(yx'-xy')\right),
$$
and it is a two--step nilpotent Lie group with underlying manifold $\R^2\times\R$. Consider the left invariant vector fields
\begin{eqnarray*}
X=\frac{\partial}{\partial x}+2y\frac{\partial}{\partial t},\quad Y=\frac{\partial}{\partial y}-2x\frac{\partial}{\partial t},
\quad T=\frac{\partial}{\partial t}.
\end{eqnarray*}
The Lie algebra of left invariant vector fields of $\h$ has a grading $\mathfrak{h} = \mathfrak{v}_1\oplus \mathfrak{v}_2$ with
\begin{displaymath}
\mathfrak{v}_1 = \mathrm{span}_{\R}\{X, Y\}\quad \text{and}\quad \mathfrak{v}_2=\mathrm{span}_{\R}\{T\}.
\end{displaymath}
The contact form $\omega$ of $\h$ is defined as the unique 1--form satisfying $X,Y\in{\rm ker}\omega$, $\omega(T)=1$. Uniqueness here is modulo change of coordinates as it follows by the Darboux Theorem. The distribution in $\h$ defined by the first layer $\mathfrak{v}_1$ is called the {\it horizontal distribution}. 
In Heisenberg coordinates $x,y,t,$ the contact form $\omega$ is given by
\begin{eqnarray*}
\omega=dt+2(xdy-ydx).
\end{eqnarray*}

There are two natural metrics defined on $\h$; the first arises from the  Kor\'anyi gauge which is given by
$$
\left|(x,y,t)\right|_\h=\left| |x+iy|^2-it\right|^{1/2}.
$$
The {\it Kor\'anyi--Cygan}  metric $d_\h$ is derived from it on $\h$, and is defined by the relation
$$
d_\h\left((x_1,y_1,t_1),\,(x_2,y_2,t_2)\right)
=\left|(x_1,y_1,t_1)^{-1}*(x_2,y_2,t_2)\right|.
$$
The {\it sub--Riemannian metric} $\langle\cdot,\cdot\rangle$ on $\h$ is given in the horizontal subbundle by the following relations:
\begin{equation}\label{eq:submetric}
 \langle X,X\rangle=\langle Y,Y\rangle=1,\quad \langle X,Y\rangle=\langle Y,X\rangle=0,
\end{equation}
and the induced norm shall be denoted by $\|\cdot\|$. The geodesics of this metric form the {\it Legendrian foliation} of $\h$ i.e. the foliation of $\h$ by horizontal curves. An (in general) absolutely continuous curve $\gamma:[a,b]\to \h$ (in the Euclidean sense) with 
\begin{displaymath}
\gamma(\tau)=(x(\tau),y(\tau),t(\tau))\in\h
\end{displaymath}
 is called {\it horizontal} if
\begin{displaymath}
 \dot{\gamma}(\tau)\in {\rm H}_{\gamma(t)}(\h)\quad \text{for almost every}\;\tau\in [a,b],
\end{displaymath}
or equivalently if
$$
\dot t(\tau)=2(y(\tau)\dot x(\tau)-x(\tau)\dot y(\tau)).
$$For a horizontal curve $\gamma$,
$$
\ell(\gamma)=\int_a^b\|\dot\gamma_h(\tau)\|d\tau=\int_a^b\sqrt{\langle\dot\gamma(\tau),X_{\gamma(\tau)}\rangle^2+\langle\dot\gamma(\tau),Y_{\gamma(\tau)}\rangle^2}d\tau
$$
and the {\it Carnot--Caratheodor\'y distance} of two arbitrary points $p,q\in\h$ is 
$$
d_{cc}(p,q)=\inf_\gamma \ell(\gamma)
$$
where $\gamma$ is a horizontal curve joining $p$ and $q$. It is proved  that the Kor\'anyi--Cygan and Carnot--Caratheodor\'y metrics generate the same infinitesimal structure and morover, the isometry groups of $(\h,d_\h)$ and $(\h,d_{cc})$ are the same.

\section{Elements of Horizontal Geometry of Surfaces in $\h$}\label{sec:elements}
In this section we define regular surfaces in the Heisenberg group  $\h$ and their horizontal normal vector field  (Section \ref{sec:surfaces}). Regular surfaces induce a contact structure from $\h$; we study this structure in Section \ref{sec:form}, in fact we comment on (local) contactomorphisms between surfaces and the horizontal flow of a regular surface (that is the foliation of the surface by horizontal surface curves). Finally, in Section \ref{sec:hormean} we define the horizontal mean curvature of a regular surface and prove that $H-$minimal regular surfaces are locally contactomorphic to the plane.
  
\subsection{Regular Surfaces--Horizontal Normal Vector Field}\label{sec:surfaces}

By a regular surface $\E$ embedded in the Heisenberg group $\h$  we shall always mean an {\it oriented regular surface of} $\R^3$, see \cite{DC}, i.e. a countable collection of surface patches $\sigma_\alpha:U_\alpha\rightarrow V_\alpha\cap\R^3$ where $U_\alpha$ and $V_\alpha$ are open sets of $\R^2$ and $\R^3$ respectively, such that
\begin{enumerate}
 \item each $\sigma_\alpha$ is a smooth homeomorphism, and
\item the differential $(\sigma_\alpha)_*:\R^2\rightarrow\R^3$ is of rank 2 everywhere.
\end{enumerate}
Let $\E:U\rightarrow\R^3$ be a regular surface and suppose that a surface patch $\sigma$ is defined in an open domain $U\subset\R^2$ by
$$
\sigma(u,v)=(x(u,v),y(u,v),t(u,v))
$$
so that its differential $\sigma_*$ is of rank 2. The tangent plane $T_\sigma(\E)$ of $\E$ at $\sigma$ is
$$
T_\sigma(\E)={\rm span}\left\{\sigma_u=\sigma_*\frac{\partial}{\partial u},\sigma_v=\sigma_*\frac{\partial}{\partial v}\right\}
$$
which is also defined by the normal vector
\begin{eqnarray*}
N_\sigma=\sigma_u\wedge\sigma_v
=\frac{\partial(y,t)}{\partial(u,v)}\frac{\partial}{\partial x}
+\frac{\partial(t,x)}{\partial(u,v)}\frac{\partial}{\partial y}
+\frac{\partial(x,y)}{\partial(u,v)}\frac{\partial}{\partial t},
\end{eqnarray*}
where $\wedge$ is the exterior product in $\R^3$. That is
$$
T_\sigma(\E)=\{V_\sigma\in T_\sigma(\R^3)\;:\;N_\sigma\cdot V_\sigma=0\}
$$
where the dot is the usual Euclidean product in $\R^3$. The unit normal vector field of $\nu_\E$ of $\E$ is uniquely defined at each local chart by the relation
$$
\nu_\sigma=\frac{\sigma_u\wedge\sigma_v}{|\sigma_u\wedge\sigma_v|}
$$ 
where $|\cdot|$ is the Euclidean norm in $\R^3$.
\medskip

\begin{defn}\label{defn-hornorspace}
Let $\E$ be a regular surface and $p\in\E$. The {\sl horizontal plane} $\mathbb{H}_p(\E)$ of $\E$ at $p$ is the horizontal plane $\mathbb{H}_p(\h)$.
\end{defn}

\medskip

For arbitrary $p\in\E$, we wish to find the relation between the horizontal  plane $\mathbb{H}_p(\E)$ and the tangent plane $T_p(\E)$. We begin by defining a suitable wedge product.

\begin{defn}
For $p\in\h$, the Heisenberg wedge product  $\wedge^\h_p$ is a mapping $T_p(\h)\times T_p(\h)\to T_p(\h)$ which assigns to each two vectors
\begin{eqnarray*}
 &&
a=a_1X+a_2Y+a_3T,\quad\text{and}\quad b=b_1X+b_2Y+b_3T
\end{eqnarray*}
of  $T_p(\h)$ the  vector $a\wedge^\h b\in T_p(\h)$ which is  given by the formal determinant
\begin{equation*}
a\wedge^\h b=\left|\begin{matrix}
                 X&Y&T\\
a_1&a_2&a_3\\
b_1&b_2&b_3     \end{matrix}\right|=\left|\begin{matrix}
a_2&a_3\\
b_2&b_3\end{matrix}\right|X+\left|\begin{matrix}
a_3&a_1\\
b_3&b_1\end{matrix}\right|Y+\left|\begin{matrix}
a_1&a_2\\
b_1&b_2\end{matrix}\right|T.
\end{equation*}
 
\end{defn}
Obviously $a\wedge^\h b=-b\wedge^\h a$ and the following clock rule holds.
\begin{equation*}\label{clock}
 X\wedge^\h Y=T,\quad Y\wedge^\h T=X, \quad T\wedge^\h X=Y.
\end{equation*}

\medskip

Thus defined, this wedge product leads to the following.

\medskip

\begin{defn}
 If $\sigma:U\rightarrow\R^3$ is a surface patch of a regular surface $\E$, the {\sl horizontal normal} $N^h_\sigma$ to $\sigma$ is the horizontal part of
$$
\sigma_u\wedge^\h \sigma_v=\sigma_*\frac{\partial}{\partial u}\wedge^\h \sigma_*\frac{\partial}{\partial v},
$$
that is
\begin{equation}\label{hornor2}
N_\sigma^h=(\sigma_u\wedge^\h \sigma_v)^h=\sigma_u\wedge^\h \sigma_v-\omega\left(\sigma_u\wedge^\h \sigma_v\right)T.
\end{equation}
The {\sl unit horizontal normal} $\nu^h_\sigma$ to $\sigma$ is
\begin{equation}\label{unithornor2}
\nu^h_\sigma=\frac{N^h_\sigma}{\|N^h_\sigma\|} 
\end{equation}
where $\|\cdot\|$ denotes the norm of the product $\langle,\cdot,\rangle$ in $\mathbb{\h}$ (recall that $\|X\|=\|Y\|=1$ and $\langle X,Y\rangle=0$). We have
\begin{equation}\label{unithornor}
 \nu^h_\sigma=\frac{(\sigma_u\wedge^\h \sigma_v)^h}{\|(\sigma_u\wedge^\h \sigma_v)^h\|}.
\end{equation}
\end{defn}

\medskip

Observe that  $N^h_\sigma$ is {\sl not} the horizontal part of $N_\sigma$. Simple calculations induce the following explicit formula:
\begin{equation}\label{defn-hornor}
N^h_\sigma=\left(\frac{\partial(y,t)}{\partial(u,v)}+2y\frac{\partial(x,y)}{\partial(u,v)}\right)X+
\left(\frac{\partial(t,x)}{\partial(u,v)}-2x\frac{\partial(x,y)}{\partial(u,v)}\right)Y.
\end{equation}

\medskip

From its very definition, it is immediately derived that the horizontal normal $N_p^h$ at a point $p\in\E$ depends on the choice of the surface patch in the following way: suppose that $(U,\sigma)$ and $(\tilde U,\tilde\sigma)$ are two overlapping patches at $p$. Then if $\Phi=\sigma^{-1}\circ\tilde\sigma$ is the transition mapping, we may find from \ref{hornor2}  that around $p$ we have
$$
N_{\tilde\sigma}^h={\rm det}(\Phi)N_\sigma^h,
$$
where ${\rm det}(\Phi)>0$ since we have already presupposed that $\E$ is oriented. At this point, we would have been ready to define the unit horizontal normal vector field in $\E$ in accordance with the unit normal vector field which is defined everywhere in a regular surface embedded into Euclidean space, but there is no assurance that a) $N_p^h\neq 0$ at all $p\in\E$ and b) $\nu^h_p$ is not in $T_p(\E)$. To this end we give the following definition.

\medskip

\begin{defn}
Let $\E$ be a regular surface. A point $p\in\E$ is called {\sl non characteristic} if $N_p^h\neq 0$. The set of characteristic points
\begin{equation*}
\mathfrak{C}(\E)=\{p\in\E\;|\;N_p^h=0\}
\end{equation*}
is called the {\sl characteristic locus} of $\E$.
\end{defn}

\medskip

By definition, the points of $\mathfrak{C}(\E)$ are given in a local chart $(U,\sigma)$ by the equations
\begin{equation*}
\frac{\partial(y,t)}{\partial(u,v)}+2y\frac{\partial(x,y)}{\partial(u,v)}=0\quad\text{and}\quad
\frac{\partial(t,x)}{\partial(u,v)}-2x\frac{\partial(x,y)}{\partial(u,v)}=0,
\end{equation*}
and therefore the Lebesgue measure of $\mathfrak{C}(\E)$ is 0 or 1. An equivalent, but not depending on coordinates  definition of the characteristic locus will be given in the next section. It remains to show that at non characteristic points of $\E$,  $\nu^h_p$ is not in $T_p(\E)$:



\medskip

\begin{prop}\label{prop-nu-nuh}
A point  $p=(x,y,t)\in\E$ is non characteristic if and only if $N_p\cdot N_p^h\neq 0$, (where the dot  denotes the Euclidean product in $\R^3$). Moreover, $N_p=N_p^h$ as vectors in $\R^3$ if and only if $x=y=0$ or $\partial(x,y)=0$. In this case,
$$
N_p=(\partial(y,t),\partial(t,x),0),\quad N_p^h=\partial(y,t)X+\partial(t,x)Y,\quad\text{and}\quad |N_p|=\| N^h_p\|
$$
and the surface at $p$ is tangent to a plane passing through $p$ which is orthogonal to the complex plane.
\end{prop}
\begin{proof}
If $p=\sigma(u,v)$, then $N_p^h$ may be written as a  vector of $\R^3$ as follows
$$
N_p^h=\left((\partial(y,t)+2y\partial(x,y)),(\partial(t,x)-2x\partial(x,y)),4(x^2+y^2)\partial(x,y)\right),
$$ 
where we have denoted $\partial(y,t)/\partial(u,v)$ by $\partial(y,t)$ etc. By taking the Euclidean dot product we find
$$
N_p\cdot N^h_p=\left(\partial(y,t)+2y\partial(x,y)\right)^2+\left(\partial(t,x)-2x\partial(x,y)\right)^2,
$$
and this vanishes if and only if $p$ is characteristic. Our second claim is immediate.
\end{proof}

\medskip

\begin{corollary}
Let $\E$ be a regular surface of $\h$. Then away from the characteristic locus, \ref{unithornor} defines a nowhere vanishing vector field $\nu^h_\E\in\mathbb{H}(\E)$, such that $\|\nu^h_\E\|=1$.
\end{corollary}

\medskip

\noindent Denote by $\J$ the complex operator acting in $\mathbb{H}(\h)$ by the relations
$$
\mathbb{J}X=Y,\quad\mathbb{J}Y=-X.
$$
The operator $\J$ acts in  the horizontal space  of a regular surface $\E$, and if $\nu^h_\E=\nu_1X+\nu_2Y$ then
$$
\J \nu^h_\E=-\nu_2X+\nu_1Y.
$$

\subsection{The Induced 1--Form. Contactomorphisms. Horizontal Flow}\label{sec:form}

Let $\E$ be a regular surface in $\h$ and denote by $\iota_\E$ the inclusion map
$
\iota_\E:\E\hookrightarrow\h,
$
given locally by a parametrisation $\sigma(u,v)=(x(u,v),y(u,v),t(u,v))$. Let $\omega=dt+2xdy-2ydx$ be the contact form of $\h$; the pullback $\omega_\E=\iota_\E^*\omega$ defines a 1--form on $\E$ which, in the local parametrisation is given by
\begin{eqnarray*}\label{omegaS}
 \omega_\E=\sigma^*\omega=(t_u+2xy_u-2yx_u) du+(t_v+2xy_v-2yx_v) dv.
\end{eqnarray*}

\medskip

\begin{prop}\label{prop-omega-locus}
The characteristic locus $\mathfrak{C}(\E)$ is the (closed) set of points of $\E$ at which $\omega_\E=0$. 
\end{prop}
\begin{proof}
We have:
\begin{eqnarray*}
&&
 \omega_\E(p)=0\quad \text{for\;some}\;p\in\E\\
&&
\Longleftrightarrow\quad \omega_p(\sigma_u)=\omega_p(\sigma_v)=0 \quad \text{for\;each\;chart}\;(U,\sigma)\;\text{containing}\;p\\
&&
\Longleftrightarrow\quad \sigma_u\;\text{and}\;\sigma_v\in\mathbb{H}_p(\E)\quad \text{for\;each\;chart}\;(U,\sigma)\;\text{containing}\;p\\
&&
\Longleftrightarrow\quad (\sigma_u\times\sigma_v)^h=0\quad \text{for\;each\;chart}\;(U,\sigma)\;\text{containing}\;p\\
&&
\Longleftrightarrow\quad p\in\mathfrak{C}(\E).
\end{eqnarray*}
\end{proof}

\medskip

Regular surfaces in $\h$ with empty characteristic locus and will be treated separately in Section \ref{sec:empty}, where we will see some consequenses of Proposition \ref{prop-omega-locus}. 

\medskip

\begin{defn}\label{def:contact}
Let $\calS$ and $\tilde\calS$ be regular surfaces and $f:\calS\to\tilde\calS$ be a smooth diffeomorphism. We may assume a weaker condition, that is we will require $f$ to be a local diffeomorphism outside the characteristic loci of $\calS$ and $\tilde\calS$. The mapping $f$ is called a local contactomorphism of $\calS$ and $\tilde\calS$ if there exists a smooth function $\lambda$ so that
$$
f^*\omega_{\tilde\calS}=\lambda\omega_\calS.
$$
\end{defn}

\medskip

Since $f$ is a local diffeomorphism, if $\sigma:U\to\R^3$ is a surface patch for $\calS$ then $\tilde\sigma=f\circ\sigma$ is  a surface patch for $\tilde\calS$ (with the possible exception of characteristic points). It follows that $f:\calS\to\tilde\calS$ is a contactomorphism if and only if
\begin{equation}\label{eq:contact-cond}
\omega_{\tilde\sigma}(u,v)=\lambda(u,v)\omega_\sigma(u,v),\quad\text{for almost all}\;\;(u,v)\in U.
\end{equation}

\medskip

 By a {\sl surface curve} on  a regular surface   $\E$ we shall always mean a smooth mapping $\gamma:I\rightarrow\E$ where $I$ is an open interval of $\R$. We wish to find conditions so that a surface curve is horizontal, i.e. its horizontal tangent $\dot\gamma_h(s)\in\mathbb{H}_{\gamma(s)}(\E)$.

\medskip

\begin{prop}\label{horsurfacecurves}
Suppose that $\sigma:U\rightarrow\h$ is a surface patch,  and $\gamma(s)=\sigma(u(s),v(s))$, $s\in I$ is a smooth surface curve (that is, $\tilde\gamma(s)=(u(s),v(s))$ a smooth curve in $U$). Then away from the characteristic locus $\gamma$ is horizontal if and only if
$$
\dot{\tilde\gamma}\in{\rm ker}\omega_\E,
$$
or in other words,
\begin{equation*}
 (t_u+2xy_u-2yx_u)\dot u+(t_v+2xy_v-2yx_v)\dot v=0
\end{equation*}
where the dot denotes $d/ds$.
In this case,
\begin{equation*}\label{hor-sur-cur}
 \dot\gamma=(x_u\dot u+x_v\dot v)X+(y_u\dot u+y_v\dot v)Y.
\end{equation*}
\end{prop}
\begin{proof}
We only prove the first statement; the other two are derived immediately. We have
\begin{eqnarray*}
\gamma\;\text{horizontal}&\Longleftrightarrow&\omega(\dot\gamma_h)=0\\
&\Longleftrightarrow&\omega(\sigma_*\dot{\tilde\gamma})=0\\
&\Longleftrightarrow&(\sigma^*\omega)(\dot{\tilde\gamma})=0\\
&\Longleftrightarrow&\dot{\tilde\gamma}\in{\rm ker}\omega_\E.
\end{eqnarray*}

\end{proof}

\medskip

The following Proposition indicates the importance of the unit horizontal normal vector field $\J\nu_\E$. 

\medskip

\begin{prop}\label{integrability-Jn}
The 1--form $\omega_\E$ defines an integrable foliation of $\E$ (with singularities at characteristic points) by horizontal surface curves. These curves are tangent to $\J\nu_\E^h$.
\end{prop}
\begin{proof}
Integrability is obvious: $\omega_\E$ is a $1-$form defined in a two--dimensional manifold. 
For the second statement, we set 
\begin{equation}\label{alphabeta}
\alpha=\frac{1}{\|N^h\|}(t_u-2yx_u+2xy_u),\quad \beta=\frac{1}{\|N^h\|}(t_v-2yx_v+2xy_v),
\end{equation}
where $\|N^h\|=\|(\sigma_u\wedge^\h\sigma_v)^h\|$,
and consider
\begin{equation}\label{JV}
J\mathcal{V}=\beta \frac{\partial}{\partial u}-\alpha \frac{\partial}{\partial v}\in{\rm ker}\omega_\E.
\end{equation}
By observing that
\begin{eqnarray*}\label{hor-sur-cur2}
&&
\beta y_u-\alpha y_v=\frac{\partial(y,t)+2y\partial(x,y)}{\|N^h\|}=\nu_1,\\
 &&\label{hor-sur-cur3}
\beta x_u-\alpha x_v=-\frac{\partial(t,x)-2x\partial(x,y)}{\|N^h\|}=-\nu_2,
\end{eqnarray*}
we obtain
\begin{eqnarray*}
 \sigma_*\mathcal(J\mathcal{V})&=&\beta\sigma_u-\alpha\sigma_v\\
&=&\beta(x_uX+y_uY+\|N^h\|\alpha T)-\alpha(x_vX+y_vY+\|N^h\|\beta T)\\
&=&(\beta x_u-\alpha x_v)X+(\beta y_u-\alpha y_v)Y\\
&=&-\nu_2X+\nu_1Y=\J\nu_\E.
\end{eqnarray*}
Note finally that by \ref{JV} the integral curves of $\J\nu_\E$ are the solutions of the system of differential equations
$$
\dot u=\beta,\quad\dot v=-\alpha.
$$
\end{proof}

\medskip

We remark for later use that when $D=\partial(x,y)\neq 0$ we also have the following expressions for $\alpha$ and $\beta$:
\begin{equation}\label{alphabeta2}
\alpha=-\frac{\nu_1x_u+\nu_2y_u}{D},\quad \beta=-\frac{\nu_1x_v+\nu_2y_v}{D}.
\end{equation}

\begin{defn}\label{horflow}
The foliation of $\E$ by the integrable curves of $\J\nu_\E$ is called the {\sl horizontal flow} of $\E$.
\end{defn}

\subsection{Horizontal Mean Curvature}\label{sec:hormean}
Horizontal mean curvature is defined as follows.

\medskip

\begin{defn}\label{hormean}
Let $\E$ be a non characteristic point of a regular surface $\E$ and let also $\nu^h_p$ $=\nu_1X+\nu_2Y$ be the  unit horizontal normal of $\E$ at $p$. The {\sl horizontal mean curvature} $H^h(p)$ of $\E$ at $p$ is given by
$$
H^h(p)=X_p\nu_1+Y_p\nu_2.
$$ 
\end{defn}

\medskip

A more geometric but equivalent definition is following by the next proposition according to which, 
the horizontal mean curvature at non characteristic  points of a regular surface may be defined  as the signed curvature of the projection to $\C$ of the leaf of the horizontal flow passing from $p$ (see also Proposition 4.24 of \cite{CDPT}).

\medskip

\begin{prop}\label{prop:mean-sec}
Let $\E$ be a regular surface and $p\in\E$ a non characteristic point. Let $\nu_\E^h=$ $\nu_1X+\nu_2Y$ be the unit horizontal normal vector field of $\E$, and  $\gamma$ the unique unit speed surface curve passing from $p$ which is  tangent  to $\J\nu_p^h$ at $p$. If $\pi=pr_\C\gamma$ is the projection of $\gamma$ on $\C$, $p'$ is the projection of $p$ and $\kappa_s$ is the signed curvature of $\pi$, then
\begin{equation*}
\kappa_s(p')=X_p\nu_1+Y_p\nu_2.
\end{equation*}
\end{prop}
\begin{proof}
 Let $p\in\E$ and $\gamma(s)$ the unit  speed horizontal surface curve passing from $p$. Let $\pi(s)$ be the projection of $\gamma(s)$ in $\C=\R^2$; then its tangent  is
\begin{equation*}
 \dot\pi(s)=(-\nu_2(s),\nu_1(s))
\end{equation*}
and its of unit speed.
We have by applying the chain rule  that
\begin{eqnarray*}
\dot \nu_1&=&(\nu_1)_x\dot x+(\nu_1)_y\dot y+(\nu_1)_t\dot t\\
&=&(X\nu_1-2yT\nu_1)(x_u\dot u+x_v\dot v)+(Y\nu_1+2xT\nu_1)(y_u\dot u+y_v\dot v)+T\nu_1(t_u\dot u+t_v\dot v)\\
&=&-(X\nu_1-2yT\nu_1)\nu_2+(Y\nu_1+2xT\nu_1)\nu_1+T\nu_1(-2y\nu_2-2x\nu_1)\\
&=&-\nu_2(X\nu_1+Y\nu_2),
\end{eqnarray*}
where we have used
$$
\dot\gamma(s)=(x_u\dot u+x_v\dot v)X+(y_u\dot u+y_v\dot v)Y=-n_2X+\nu_1Y,
$$
and the relation  $\nu_1Y\nu_1=-\nu_2Y\nu_2$ which follow from $\nu_1^2+\nu_2^2=1$.
Working analogously for $\dot \nu_2$ (using $\nu_1X\nu_1=-\nu_1X\nu_2$ this time), we have
$
\dot \nu_2=\nu_1(X\nu_1+Y\nu_2),
$
hence
\begin{eqnarray*}
\ddot\pi=(-\dot\nu_2,\dot\nu_1)&=&(-\nu_1(X\nu_1+Y\nu_2),-\nu_2(X\nu_1+Y\nu_2))\\
&=&\kappa_s(-\nu_1,-\nu_2)
\end{eqnarray*}
where $\kappa_s$ is the signed curvature of the curve $\pi$. This yields
$\kappa_s=X\nu_1+Y\nu_2$. 

\end{proof}

\medskip

A local expression for $H^h$ is in order:

\medskip

\begin{prop}\label{prop:main}
Let $\E$ be a regular surface in $\h$. In every surface patch $\sigma=(x,y,t)$ with $\partial(x,y)\neq 0$ and sufficiently away from the characteristic locus, the horizontal mean curvature is given by
 \begin{equation}
  H^h(\sigma)=\frac{\partial(\nu_1,y)+\partial(x,\nu_2)}{\partial(x,y)},
 \end{equation}
where $\nu_i$, $i=1,2$ are the components of the unit horizontal normal vector $\nu$ of $\E$. If $\partial(x,y)=0$, then $H^h(\sigma)=0$.
\end{prop}

\begin{proof}
 Suppose first that $\partial(x,y)\neq 0$. Using the chain rule we write
 \begin{eqnarray*}
  &&\label{eq:H1}
  (\nu_1)_u=(\nu_1)_xx_u+(\nu_1)_yy_u+(\nu_1)_tt_u=(X\nu_1)x_u+(Y\nu_1)y_u+(t_u-2yx_u+2xy_u)T\nu_1,\\
 &&\label{eq:H2}
  (\nu_2)_u=(\nu_2)_xx_u+(\nu_2)_yy_u+(\nu_2)_tt_u=(X\nu_2)x_u+(Y\nu_2)y_u+(t_u-2yx_u+2xy_u)T\nu_2,\\
  &&\label{eq:H3}
  (\nu_1)_v=(\nu_1)_xx_v+(\nu_1)_yy_v+(\nu_1)_tt_v=(X\nu_1)x_v+(Y\nu_1)y_v+(t_v-2yx_v+2xy_v)T\nu_1,\\
 &&\label{eq:H4}
  (\nu_2)_v=(\nu_2)_xx_v+(\nu_2)_yy_v+(\nu_2)_tt_v=(X\nu_2)x_v+(Y\nu_2)y_v+(t_v-2yx_v+2xy_v)T\nu_2.
  \end{eqnarray*}
The first and the third equation are written as
\begin{eqnarray*}
 &&
 (X\nu_1)x_u+(Y\nu_1)y_u=(\nu_1)_u-\alpha\|N^h\| T\nu_1,\\
 &&
 (X\nu_1)x_v+(Y\nu_1)y_v=(\nu_1)_v-\beta\|N^h\| T\nu_1
\end{eqnarray*}
where we have used Equations \ref{alphabeta}. Solving the system we obtain
$$
X\nu_1=\frac{\partial(\nu_1,y)+\|N^h\|\nu_1T\nu_1}{\partial(x,y)},\quad Y\nu_1=\frac{\partial(x,\nu_1)+\|N^h\|\nu_2T\nu_1}{\partial(x,y)},
$$
where we have used Equations \ref{alphabeta2}. 
 In an analogous manner, we obtain the following from the second and the fourth equations:
 $$
 X\nu_2=\frac{\partial(\nu_2,y)+\|N^h\|\nu_1T\nu_2}{\partial(x,y)},\quad Y\nu_2=\frac{\partial(x,\nu_2)+\|N^h\|\nu_2T\nu_2}{\partial(x,y)}.
 $$
 Therefore
 \begin{eqnarray*}
 X\nu_1+Y\nu_2&=& \frac{\partial(\nu_1,y)+\partial(x,\nu_2)+\|N^h\|(\nu_1T\nu_1+\nu_2T\nu_2)}{\partial(x,y)}\\
 &=&\frac{\partial(\nu_1,y)+\partial(x,\nu_2)}{\partial(x,y)},
 \end{eqnarray*}
 since $\nu_1^2+\nu_2^2=1$ and hence $\nu_1T\nu_1+\nu_2T\nu_2=0$.
 
Finally if $\partial(x,y)=0$, then from Proposition \ref{prop-nu-nuh} it is deduced that the horizontal normal vector field $\nu^h_\sigma$ is orthogonal to a plane vertical to the complex plane and the image of $\sigma$ belongs to that plane. Thus $\mathbb{J}\nu^h_\sigma$ is tangent to the plane and the horizontal flow comprises of straight lines. The proof is complete. 
\end{proof}

\medskip

\begin{prop}\label{prop:contact-minimal}
If a regular surface $\E$ in $\h$ is locally contactomorphic to the complex plane, then it is $H-$minimal.
\end{prop}

\begin{proof}
First we prove the statement for graphs $G_f$ of smooth functions $t=f(x,y)$ over $\C$. Here $(x,y)$ lie in an open subset of the plane. Let
$$
\sigma(x,y)=(x,y,f(x,y)),\quad (x,y)\in U.
$$
The induced $1-$form is $\omega_{G_f}=(f_x-2y)dx+(f_y+2x)dy$. From the contactomorphism condition we also have
\begin{equation*}
f_x-2y=-2\lambda y,\quad\text{and}\quad f_y+2x=2\lambda x
\end{equation*} 
for some non zero function $\lambda$. Moreover,
$$
N^h=(-f_x+2y)X+(-f_y-2x)Y=2\lambda(yX-xY),
$$
and therefore 
$$
\nu_{G_f}=\nu_1X+\nu_2Y=\pm\frac{yX-xY}{(x^2+y^2)^{1/2}}.
$$
Using Proposition \ref{prop:main} we have for the positive sign case (the other case is treated analogously):
\begin{eqnarray*}
H^h&=&\partial(\nu_1,y)+\partial(x,\nu_2) \quad (\partial(x,y)=1),\\
&=&\partial\left(\frac{y}{(x^2+y^2)^{1/2}},y\right)+\partial\left(\frac{x}{(x^2+y^2)^{1/2}},x\right)\\
&=&y\partial\left(\frac{1}{(x^2+y^2)^{1/2}},y\right)+x\partial\left(\frac{1}{(x^2+y^2)^{1/2}},x\right)\\
&=&y\partial_x\left(\frac{1}{(x^2+y^2)^{1/2}}\right)-x\partial_y\left(\frac{1}{(x^2+y^2)^{1/2}}\right)\\
&=&y\cdot\frac{-x}{(x^2+y^2)^{3/2}}-x\cdot\frac{-y}{(x^2+y^2)^{3/2}}\\
&=&0.
\end{eqnarray*}
Next we show that all coordinate planes are locally contactomorphic; we will treat the case of the planes $x=0$ and $t=0$ and leave the other cases as an exercise. We  parametrise the plane $x=0$ by $\sigma(u,v)=(0,u,v)$ and consider the map $f:\{x=0\}\to \{t=0\}$ given by
$$
(0,u,v)\mapsto(uv,v,0).
$$ 
Denote by $\tilde\sigma$ the surface patch $f\circ\sigma$. Then
$$
\omega_\sigma=dv\quad\text{and}\quad\omega_{\tilde\sigma}=-2u^2 dv
$$
which by the contact condition \ref{eq:contact-cond} proves our assertion.

If now 
$
\sigma(u,v)=(x(u,v),y(u,v),t(u,v)
$
is an arbitrary surface patch for $\E$, from regularity we have that at least one of $\partial(x,y)$, $\partial(y,t)$ and $\partial(t,x)$ is different from zero. We may now assume that $\partial(x,y)\neq 0$ and reparametrise if necessary by
$$
\tilde u=x(u,v),\quad \tilde v=y(u,v)
$$
to obtain the regular surface patch $\sigma(\tilde u,\tilde v, t(\tilde u,\tilde v))$ which is a local graph of a function over the complex plane.
\end{proof}

\section{Straight Ruled Surfaces}\label{sec:ruled}

In this section we define straight ruled surfaces in $\h$ and prove Theorem \ref{thm:main}. For the proof, we use two different ways to show that straight ruled surfaces are $H-$minimal; the first one is by showing that they are locally contactomorphic to the complex plane and the second is straightforward.

A straight ruled surface in $\h$ is a surface  which is formed by a union of straight lines (the rulings of the surface),  in the following manner. Suppose that $\gamma=\gamma(s)$, where $s$ lies in an open interval $I$ of $\R$, is a (not necessarily horizontal) smooth curve and $V=V(s)$ is a unit horizontal vector field along $\gamma$, i.e. $V(s)\in{\mathbb H}_{\gamma(s)}(\h)$. For reasons that will be justified below, we assume that the projected curve $pr_\C(\gamma)$ is not a straight line. At any point $q\in\gamma$, say $q=\gamma(s)$ we consider the straight line passing from $q$ in the direction of $V(s)$. Then a point $p$ on the straight line satisfies 
$
p=\gamma(s)+vV(s)
$ 
for some $v$. The {\it straight ruled surface} $\calR(\gamma)$ is the union of all such straight lines, therefore it admits a parametrisation by the (single) surface patch $\sigma:I_s\times\R\to\R^3$ where $I_s$ is an open interval of $\R$ and
$$
\sigma(s,v)=\gamma(s)+vV(s).
$$
If $\gamma=(x,y,t)$ and $V=aX+bY$, $a^2+b^2=1$, we write
\begin{eqnarray*}
\sigma(s,v)&=&\left(\tilde x(s,v),\tilde y(s,v),\tilde t(s,v)\right)\\
&=&\left(x(s)+v a(s),y(s)+v b(s), t(s)+2v(y(s)a(s)-x(s)b(s)\right),
\end{eqnarray*}
and calculate (denoting $d/ds$ by dot)
\begin{eqnarray*}
&&
\tilde x_s=\dot x+v\dot a,\quad \tilde y_s=\dot y+v\dot b,\quad \tilde t_s=\dot t+2v(\dot y a+y\dot a-\dot x b-x\dot b)\\
&&
\tilde x_v=a,\quad \tilde y_v=b,\quad \tilde t_v=2(ya-xb).
\end{eqnarray*}
Regularity: $\sigma$ has to be a regular surface patch. Set
$$
\delta(s)=(a(s),b(s),2\left(y(s)a(s)-x(s)b(s)\right).
$$
Since $\sigma_s=\dot\gamma+v\dot\delta$ and $\sigma_v=\delta$, $\sigma$ is regular if $\dot\gamma+v\dot\delta$ and $\delta$ are linearly independent. For example, this happens if
$$
\left(\dot x(s),\dot y(s),\dot t(s)\right)\quad\text{and}\quad\delta(s)
$$
are linearly independent and $v$ is sufficiently small. Thus regularity is assured if $V(s)$ is never tangent to $\gamma$.
\begin{eqnarray*}
&&
\sigma_s=(\dot x+v\dot a)X+(\dot y+v\dot b)Y+\left(\dot t+4v(a\dot y-b\dot x)+2(x\dot y-y\dot x)+2v^2(a\dot b-b\dot a)\right)T,\\
&&
\sigma_v=aX+bY=V,
\end{eqnarray*}
and
$$
\left(\sigma_s\wedge^h\sigma_v\right)^h=\eta\left(-bX+aY\right)=\eta\mathbb{J}V,
$$ 
where
$$
\eta=\eta(s,v)=\dot t(s)+2(x(s)\dot y(s)-y(s)\dot x(s))+4v(a(s)\dot y(s)-b(s)\dot x(s))+2v^2(a(s)\dot b(s)-b(s)\dot a(s)).
$$
Thus the characteristic locus is
$$
\mathfrak{C}(\calR(\gamma))=\{(s,v)\in I_s\times I_v:\;\eta(s,v)=0\},
$$
where $I_v$ is an appropriately small open interval of $\R$.
The exceptional case when $\eta$ vanishes identically occurs when the projection $pr_\C(\gamma)$ is a straight line. This can be seen as follows. The function $\eta$  is a quadratic polynomial in $v$ therefore it vanishes identically if and only if the following relations hold simultaneously:
\begin{eqnarray*}
&&
\dot t(s)=2(y(s)\dot x(s)-x(s)\dot y(s)),\\
&&
a(s)\dot y(s)-b(s)\dot x(s)=0,\\
&&
a(s)\dot b(s)-b(s)\dot a(s)=0.
\end{eqnarray*}
From the first relation it follows that $\gamma$ has to be horizontal; from the second we have that $V$ is parallel to the horizontal tangent $\dot\gamma=\dot xX+\dot yY$ and since $V$ has been supposed to be unit, we have $V=\pm\dot \gamma$. Then the third relation reads
$$
\pm(\dot x\ddot y-\dot y\ddot x)=0.
$$
But the left hand side is (up to sign) equal to the signed curvature of the projected curve $pr_\C(\gamma)$. Hence  $pr_\C(\gamma)$ has to be a straight line, which contradicts our assumptions for $\calR(\gamma)$. (Note that in this case there is no surface defined). Another special case occurs when $\gamma$ is horizontal; then $\dot t+2(x\dot y-y\dot x)=0$ and thus the characteristic locus includes all points of $\gamma$. 

\medskip

\begin{prop}\label{prop:straight-contact}
Any straight ruled surface $\calR(\gamma)$ is locally contactomorphic to the complex plane $\C$ and thus is $H-$minimal.
\end{prop}
\begin{proof}
We only have to prove our first statement; the second follows from Proposition \ref{prop:contact-minimal}.
Let $\gamma=(x,y,t)$ and $V=aX+bY$ as before and also
$$
\sigma(s,v)=(\tilde x(s,v),\tilde y(s,v),\tilde t(s,v))
$$
be the surface patch for $\calR(\gamma)$.  Then
\begin{eqnarray*}
\omega_{\calR(\gamma)}&=&\sigma^*\omega\\
&=&(\tilde t_s+2\tilde x\tilde y_s-2\tilde y\tilde x_s)ds+(\tilde t_v+2\tilde x\tilde y_v-2\tilde y\tilde x_v)dv\\
&=&\eta ds.
\end{eqnarray*}
We now consider the following local parametrisation for $\C$:
$$
\tilde\sigma(s,v)=(a(s)v,b(s)v,0).
$$
Under this parametrisation, $\C$ is trivially a straight ruled surface; the curve $\gamma$ is the single point $(0,0,0)$, and the horizontal flow comprises of the straight lines passing through the origin.
Then
\begin{eqnarray*}
&&
\omega_{\C}=\tilde\sigma^*\omega
=2v^2(a\dot b-b\dot a)ds
=(2v^2(a\dot b-b\dot a)/\eta) \omega_{\calR(\gamma)}.
\end{eqnarray*}

\end{proof}

\begin{rem} Here is a straightforward proof of $H-$minimality of straight ruled surfaces. The unit horizontal vector field is
$$
\nu=\nu_1X+\nu_2Y=\pm(bX-aY).
$$
We suppose first that $\nu_1=b$ and $\nu_2=-a$; the other case is treated similarly. We find
$$
(\nu_1)_s=\dot b,\quad (\nu_1)_v=0,\quad (\nu_2)_s=-\dot a,\quad (\nu_2)_v=0.
$$
Using Proposition \ref{prop:main} we have at non characteristic points
$$
H^h=\frac{\partial(\nu_1,\tilde y)+\partial(\tilde x,\nu_2)}{\partial(\tilde x,\tilde y)}=\frac{b\dot a+a\dot b}{b\dot x-a\dot y+v(b\dot a-a\dot b)}=0,
$$
since $a^2+b^2=1$.



We  also stress here that it is geometrically clear that the parametric lines  $s={\rm const.}$ are the horizontal flow of $\calR(\gamma)$. This can also be seen by solving the system of equations \ref{hor-sur-cur2} and \ref{hor-sur-cur3} to obtain $\beta=0$.
\end{rem}

\noindent{\it Proof of the Main Theorem \ref{thm:main}.}
The first statement of the Theorem follows from Proposition 
 \ref{prop:straight-contact}. For the second statement, let first $\E$ be a regular surface and $p\in\E$ be a non characteristic point. Since the horizontal flow foliates $\E$ by horizontal surface curves $\gamma_s$ of unit horizontal speed tangent to $\mathbb{J}\nu_\E$, $s\in I$,  consider the integral curve $\gamma_{s_0}(v)$ passing from $p$, where $v$ lies in a sufficiently small interval: $\gamma_{s_0}(v_0)=p$ for some $v_0$ in that interval. There exists an open subset $U$ of $\R^2$, with $(s_0,v_0)\in U$ and a smooth mapping $\sigma:U\to\E$ so that 
$$
\sigma(s,v)=\gamma_s(v),\quad (s,v)\in U
$$
and we may shrink $U$ so that it does not contain any characteristic points. Suppose now that $\E$ has zero horizontal mean curvarure; by Proposition \ref{prop:mean-sec}, the curves $pr_\C(\gamma_s)$ have zero signed curvature, therefore they are pieces of straight lines. It follows that if
$$
\sigma(s,v)=(x_s(v),y_s(v), t_s(v)),
$$
then we have $d^2x_s/dv^2=d^2y_s/dv^2=0$ and thus
$$
x_s(v)=a(s)v+x(s),\quad y_s(v)=b(s)v+y(s),
$$
for some smooth functions $x,y,a,b$. Since $\gamma_s$ has unit horizontal speed, we have $a^2+b^2=1$ and since it is horizontal, we also have
\begin{eqnarray*}
 \frac{dt_s}{dv}&=&2\left(y_s\frac{dx_s}{dv}-x_s\frac{dy_s}{dv}\right)\\
 &=&2\left((y+b)a-(x+a)b\right)\\
 &=&2(ya-xb),
\end{eqnarray*}
and therefore $t_s(v)=2v\left(y(s)a(s)-x(s)b(s)\right)+t(s)$, where $t(s)$ is a smooth function of $s$. Therefore the patch $\sigma$ above is a patch of a piece of a straight ruled surface. Since our point $p$ is arbitrary, we conclude  the Theorem. \hfill$\Box$

\section{Regular Surfaces in $\h$ with Empty Characteristic Locus}\label{sec:empty}
In this section we give two examples of regular surfaces $\E$ with empty $\mathfrak{C}(\E)$. First, we examine horizontal tangent developables which comprise of the counterparts of tangent developables in the Euclidean case. Secondly, we show that surfaces with empty characteristic locus and closed induced $1-$form can be only generalised cylinders which have constant horizontal mean curvature. An arbitrary generalised cylinder is not a straight ruled surface; this happend only in the case of a plane orthogonal to $\C$. Two indicative examples of surfaces are given in the end of this section. The first, that of the hyperbolic paraboloid shows that there exists a developable Euclidean surface with negative Gaussian curvature which is also a straight ruled surface. The second, that of the cone, shows that a Euclidean cone, although having empty characteristic locus and zero Gausian curvature, can not be a straight ruled surface.

We start with the following proposition which is a direct consequence of Proposition \ref{prop-nu-nuh}. 

\medskip

\begin{prop}\label{prop-locus-closed}
Let $\E$ be an oriented regular surface curve. Then the following are equivalent:
\begin{enumerate}
 \item The characteristic locus $\mathfrak{C}(\E)$ of $\E$ is the null set.
\item The induced 1--form $\omega_\E$ is nowhere zero.
\item The characteristic locus $\mathfrak{C}(\E)$ of $\E$ is the null set, if and only if the horizontal flow has no singularities. 
\end{enumerate}
\end{prop}

\subsection{Horizontal tangent developables} \label{sec:hor-developables}

Let $\gamma$ be a {\it horizontal} curve parametrised so that it is of unit horizontal speed, that is
$$
\gamma(s)=(x(s),y(s),t(s)),\quad \dot t(s)=2(y(s)\dot x(s)-x(s)\dot y(s)),
$$
and $\dot x(s)^2+\dot y(s)^2=1$ for $s$ in an open interval $I$ of $\R$. We also suppose that $\gamma$ is not a straight line; for $\gamma$ horizontal $pr_\C(\gamma)$ is a straight line if and only if $\gamma$ is a straight line. 
The surface $\calT(\gamma)$ of horizontal tangent developables of $\gamma$ is defined by the single surface patch
$$
\sigma(s,v)=\gamma(s)+v\dot\gamma(s).
$$
Regularity: Since
\begin{eqnarray*}
&&
\sigma_s=(\dot x+v\ddot x,\dot y+v\ddot y, \dot t+v\ddot t),\\
&&
\sigma_v=(\dot x,\dot y, \dot t),
\end{eqnarray*}
we have $\sigma_s\wedge\sigma_v=v\ddot\gamma\wedge\dot\gamma$. Hence, in the first place, the (usual) curvature $\kappa(\gamma)$ of $\gamma$ has to be positive everywhere. Since $\gamma$ is horizontal,
$$
\kappa(\gamma)=\left|(\ddot x,\ddot y, 2(y\ddot x-x\ddot y)\right|,
$$
which vanishes only if $\ddot x=\ddot y=0$, i.e.  only if $\gamma$ is a straight line. Moreover, we have to exclude the points of $\gamma$ since at these points $v=0$.

Thus defined, $\calT(\gamma)$ is a special case of a straight ruled surface (here $V(s)=\dot\gamma(s)$ the unit horizontal tangent of $\gamma$) and therefore it is locally contactomorphic to the plane $\C$ and  has vanishing horizontal mean curvature. Note that the characteristic locus of $\calT(\gamma)$ is empty, since we have assumed regularity for  $\calT(\gamma)$.   

\subsection{Surfaces with empty characteristic locus and closed induced form}
Below we  trace all regular oriented surfaces $\E$ in $\h$ with empty characteristic locus and with the additional property  that $\omega_\E$ is closed.

\medskip

\begin{prop}\label{cylinders}
Regular  surfaces $\E$ in $\h$ with empty characteristic locus and closed induced 1--form $\omega_\E$ are exactly the  Euclidean generalised cylinders which are obtained by translating a regular curve lying in the complex plane $\C$  along the vector field $T$.
\end{prop}

\begin{proof}
If $\sigma:U\to\E$, $\sigma=(x,y,t)$ is an arbitrary surface patch for $\E$, then $d\omega_\E=0$ induces $\partial(x,y)=0$. From Proposition \ref{prop-locus-closed} we see that if $\E$ is such a surface, then for every parametrisation $\sigma$ we have
$$
\sigma_u\times\sigma_v\perp T=\frac{\partial}{\partial t}
$$ 
as vectors in $\R^3$. But this is equivalent to say that either $\sigma_u=\rho(u,v)\partial_t$ or $\sigma_v=\rho^*(u,v)\partial_t$ where $\rho$ and $\rho^*$ are smooth functions of $(u,v)$. Suppose the first holds; the second case is treated analogously. We obtain
$$
\sigma(u,v)=\left(x(v),y(v),\int_{u_0}^u\rho(\xi,v)d\xi\right),
$$
and we may reparametrise by 
$$
\tilde u=\int_{u_0}^u\rho(\xi,v)d\xi,\quad\tilde v=v,
$$
to obtain
$$
\sigma(\tilde u,\tilde v)=(x(\tilde v),y(\tilde v),\tilde u).
$$
Since $\E$ is regular, condition $\sigma_{\tilde u}\wedge\sigma_{\tilde v}=(\dot y\rho,\dot x\rho,0)\neq 0$, where the dot stands for $d/d\tilde v$, is equivalent to that the curve $\gamma(\tilde v)=(x(\tilde v),y(\tilde v),0)$ is regular.
Thus
$$
\sigma(\tilde u,\tilde v)=\gamma(\tilde v)+\tilde u\partial_t.
$$
The proof is complete.
\end{proof}

\medskip

\begin{prop}\label{prop:closed}
The only regular  surfaces $\E$ in $\h$ with empty characteristic locus, closed induced $1-$form $\omega_\E$ and constant horizontal mean curvature are
\begin{enumerate}
 \item the planes which are perpendicular to $\C$; these have $H^h\equiv 0$ and
\item the right cylinders whose profile curve is a circle of radius $R$; these have  $H^h\equiv 1/R$. 
\end{enumerate}
\end{prop}
\begin{proof}
Let
$$
\sigma(u,v)=(x(v),y(v),u)
$$
a generalised cylinder. Since $\gamma(v)$ is regular we may reparametrise so that it has unit speed, $\dot x(v)^2+\dot y(v)^2=1$. Then,
\begin{equation*}
\nu^h=\dot y X-\dot x Y,\quad \mathbb{J}\nu^h=\dot x X+\dot y Y,
\end{equation*}
and therefore the horizontal flow is comprising of all horizontal lifts of $\gamma$. Thus
$
H^h=\kappa_s(\gamma),
$
where $\kappa_s(\gamma)$ is the signed curvature of $\gamma$. 
\end{proof}

\subsection{Two examples: Euclidean Ruled Surfaces vs. Straight Ruled Surfaces}

\subsubsection*{Hyperbolic paraboloid} The hyperbolic paraboloid $z=y^2-x^2$ is a doubly ruled surface in the usual sense and a straight ruled surface $\calR(\gamma)$ where $\gamma$ is the parabola $z=y^2$:
$$
\sigma(s,v)=(0,s,s^2)+\frac{1}{\sqrt{2}}v(X+Y).
$$
Its characteristic locus is the plane $x+y=0$. Recall that as a surface in $\R^3$ it has negative Gaussian curvature; however, since it is a straight ruled surface in $\h$ it has zero horizontal mean curvature.

\subsubsection*{Cone} The cone $x^2+y^2=z^2$ is a Euclidean ruled surface with zero Gaussian curvature. On the other hand, as a regular surface in $\h$ it has empty characteristic locus (the origin is not a regular point for the cone) and non zero (actually non constant) horizontal mean curvature. To see this, parametrise the lower part of the cone by
$$
\sigma(u,v)=(u\cos v,u\sin v, u), \quad u<0,\; v\in (0,2\pi).
$$
One finds
$$
\nu^h=\frac{(\cos v-2u\sin v)X+(\sin v+2u\cos v)}{(1+4u^2)^{1/2}}
$$
and
$$
H^h=\frac{1}{u(1+4u^2)^{3/2}}.
$$
Thus it is not a straight ruled surface in $\h$.


\begin{thebibliography}{ZZ99}

 \bibitem{CDPT} L.~Capogna \& D.~Danielli \& S.D.~Pauls \& J.T.~Tyson;
 {\sl An introduction to the Heisenberg group and the sub--Riemannian isoperimetric problem}.
 Progress in Mathematics, {\bf 259}. Birkhäuser Verlag, Basel, 2007. 
 
 \bibitem{CHMY}  J-H.~Cheng \& J--F.~Hwang \& A.~Malchiodi \& P.~Yang;
 {\sl Minimal surfaces in pseudohermitian geometry}.
 Ann. Sc. Norm. Super. Pisa Cl. Sci. {\bf 5} 4 (2005), no. 1, 129--177. 

\bibitem{DC}
M.P.~Do Carmo;
{\sl Riemannian geometry}.
Birkh\"auser, XVI ed. (1992).

\bibitem{DGNP1} D.~Danielli \& N.~Garofalo \& D-M.~Nhieu \& S.D.~Pauls;
{\sl Instability of graphical strips and a positive answer to the Bernstein problem in the Heisenberg group H1}.
J. Diff. Geom. {\bf 81} (2009), no. 2, 251--295. 

\bibitem{DGNP2} 
D.~Danielli \& N.~Garofalo \& D-M.~Nhieu \& S.D.~Pauls;
{\sl The Bernstein problem for embedded surfaces in the Heisenberg group H1}.
Indiana Univ. Math. J. {\bf 59} (2010), no. 2, 563--594.


\bibitem{DGNP3} 
D.~Danielli \& N.~Garofalo \& D-M.~Nhieu \& S.D.~Pauls;
{\sl Stable complete embedded minimal surfaces in H1 with empty characteristic locus are vertical planes}.
Arxiv.math/0903.4296v1 [math.DG]




\bibitem{GP}
N.~Garofalo \& S.D.~Pauls;
{\sl The Bernstein problem in the Heisenberg group}.
Arxiv.math/0209065 [math.DG]


\bibitem{G} W.~Goldman;
{\sl Complex hyperbolic geometry}.
Clarendon Press, Oxford, (1999).

\bibitem{I} J-I.~Inoguchi; 
{\sl Minimal surfaces in the 3-dimensional Heisenberg group}.
Differ. Geom. Dyn. Syst. {\bf 10} (2008), 163--169.


\bibitem{KR} A.~Kor\'anyi \& H.M.~Reimann;
{\sl Foundations for the theory of quasiconformal mappings of the Heisenberg group}.
Adv. in Math. {\bf 111} (1995), 1--87.



 
\bibitem{P} S.D.~Pauls; 
{\sl Minimal surfaces in the Heisenberg group}.
Geom. Ded. {\bf 104} (2004), 201--231. 


\end{thebibliography}
\end{document}